\documentclass[11pt,paper=letter,abstract=true]{scrartcl}

\usepackage[page,center,header,footer,titling,theorem]{cjquines}

\usepackage{enumitem}
\usepackage{authblk}

\newcommand{\theauthor}{C. J. Quines, M. Sun}

\begin{document}

\title{Bounds on metric dimension\\for families of planar graphs}
\renewcommand{\thetitle}{Bounds on metric dimension for families of planar graphs}

\author[1]{Carl Joshua Quines\thanks{Corresponding author. E-mail address: \mailto{cjquines0@gmail.com}.}}
\author[2]{Michael Sun\thanks{E-mail address: \mailto{michael03px2018@sasstudent.org}.}}
\affil[1]{Valenzuela City School of Mathematics and Science, Valenzuela, Philippines}
\affil[2]{Shanghai American School, Shanghai, China}

\date{March 11, 2017}

\maketitle

\begin{abstract}
  The concept of metric dimension has applications in a variety of fields, such as chemistry, robotic navigation, and combinatorial optimization. We show bounds for graphs with $n$ vertices and metric dimension $\beta$. For Hamiltonian outerplanar graphs, we have $\beta \leq \ceil{\frac{n}2}$; for outerplanar graphs in general, we have $\beta \leq \floor{\frac{2n}{3}}$; for maximal planar graphs, we have $\beta \leq \floor{\frac{3n}{4}}$. We also show that bipyramids have a metric dimension of $\floor{\frac{2n}{5}} + 1$. It is conjectured that the metric dimension of maximal planar graphs in general is on the order of $\floor{\frac{2n}{5}}$.

\end{abstract}

\section{Introduction}

Consider a robot navigating through a graph, which wants to know its current vertex. It sends a signal to each of several landmarks, which send a signal back in turn, so the robot knows its distance to each landmark.

If the set of landmarks is well-chosen, then given the distances to each of the landmarks, the robot can uniquely determine which vertex it is in. We call such a set of landmarks a resolving set. The metric dimension of the graph is the number of landmarks in the smallest resolving set.

The metric dimension of a graph is well-studied. It was first initiated by Harary and Melter in 1976 \cite{H}, and discovered independently by Slater in 1988 \cite{S}. This invariant has applications in robotics \cite{K} and chemistry \cite{J}.

Much progress has been made in studying the metric dimensions of specific graph families: trees, grids \cite{K}, paths, cycles, complete graphs, wheels, some hypercubes \cite{C}, among other results. In this paper, we take a more general direction by finding bounds on the metric dimension for certain families of planar graphs.

We now give formal definitions of some terms. For the rest of the paper, we assume all graphs are finite, simple, and connected. Let $G$ be such a graph. We say a vertex $u$ in $G$ \emph{resolves} two other vertices $v$ and $w$ in $G$ if the distance from $u$ to $v$ is not equal to the distance from $u$ to $w$. A subset of vertices of $G$ is a \emph{resolving set} of $G$ if there exists a vertex in the set that resolves each pair of vertices in $G$. The \emph{metric dimension} of $G$ is the smallest cardinality of a resolving set, which we will denote by $\beta(G)$, or simply $\beta$ when the context is clear.

A graph is said to be \emph{planar} if there exists a drawing of it on the plane such that no two edges intersect except at the vertices. Planar graphs have \emph{faces}, which are contiguous regions enclosed by edges. We call a planar graph \emph{maximal planar} if all the faces of the graph are enclosed by three edges. We call a planar graph \emph{outerplanar} if it can be drawn such that all its vertices are on the outer face. A graph is \emph{Hamiltonian} if there exists a cycle that goes through all vertices exactly once.

The main result of the paper is the following:

\begin{theorem}
  \label{main}
  For the following families of graphs with $n$ vertices and metric dimension $\beta$:
  \begin{enumerate}[label=(\alph*),ref=\ref{main}\alph*]
    \item\label{first} outerplanar hamiltonian graphs satisfy $\beta \leq \ceil{\frac{n}{2}}$.
    \item\label{second} outerplanar graphs satisfy $\beta \leq \floor{\frac{2n}{3}}$.
    \item\label{third} maximal planar graphs satisfy $\beta \leq \floor{\frac{3n}{4}}$.
  \end{enumerate}
\end{theorem}

\section{Selection through neighbors}

Recall that the \emph{neighbors} of a vertex are the vertices adjacent to it. We say two vertices are \emph{alike} if they have the same neighbors. Note that two adjacent vertices can never be alike, as no vertex is its own neighbor.

At this point, we provide some new definitions. We begin with a set of vertices of a graph $G$, which we call the \emph{working set}. To \emph{select} a vertex would be to add it to the working set. We say that a vertex $u$ in $G$ is \emph{resolved} if, for all other vertices $v$ in $G$, there exists a vertex in the working set that resolves $u$ and $v$.

\begin{lemma}
  To resolve a vertex, it is necessary to select all vertices alike to it.
\end{lemma}

\begin{proof}
  Consider a pair of alike vertices. Since they share the same neighbors, it follows that they have the same distances to all other vertices in the graph. Thus, none of the other vertices in the graph can resolve this pair. It is then necessary to select either vertex.
\end{proof}

We note that alikeness is an equivalence relation: it is reflexive, symmetric, and transitive. It then partitions the vertices of a graph into several equivalence sets. We denote by $s$ the number of these equivalence sets. Applying the above lemma gives an immediate lower bound on the metric dimension:

\begin{corollary}
  For a graph with $n$ vertices and metric dimension $\beta$, partitioned by alikeness to $s$ equivalent sets, we have $\beta \geq n-s.$
\end{corollary}

The crux idea of this section is the following lemma. It gives the ability to construct a resolving set by resolving each vertex through selecting its neighbors.

\begin{lemma}
  \label{selection}
  For a vertex to be resolved, it is sufficient to either:
  \begin{itemize}
    \item select that vertex in the resolving set, or
    \item select all its neighbors and all vertices alike to it.
  \end{itemize}
\end{lemma}

\begin{proof}
  It is clear that any selected vertex $v$ resolves itself: no other vertex in the graph has distance zero to $v$. If all of the neighbors and alike vertices to $v$ are selected, then it is resolved: any vertex with the same neighbor set by $v$ has to be alike to it, and all these vertices are resolved as well.
\end{proof}

We now prove the first main result, Theorem \ref{first}.

\begin{proof}[Proof of Theorem \ref{first}]
  Label the vertices of the Hamiltonian cycle as $v_1, v_2, \ldots, v_n$ and select all $v_i$ such that $i$ is odd. This selects at most $\ceil{\frac{n}2}$ vertices. We claim that this is a resolving set for the graph.

  It is clear by the Lemma \ref{selection} that each selected vertex is resolved. Suppose for the sake of contradiction there exists two vertices which are not resolved, namely $v_i, v_j$. If they are not resolved, then they must share the same neighbors. We have two cases:

  \begin{itemize}
    \item $\abs{i-j}=2$. If this is the case, $v_j$ must be adjacent to $v_{i-1}$ and $v_i$ must be adjacent to $v_{j+1}.$ It is then clear, through contracting the appropriate paths, that $K_4$ would be a minor of the graph, contradicting outerplanarity.

    \item $\abs{i-j}>2$. It must be the case that $v_{j-1}$ and $v_{j+1}$ are adjacent to $v_i$. Once again, $K_4$ would be a minor of the graph, contradicting outerplanarity.\qedhere
  \end{itemize}
\end{proof}

\section{Restricting alike vertices}

We present two lemmas restricting the pairs of alike vertices in outerplanar and maximal planar graphs. These will be used in the subsequent section to produce bounds on their metric dimension.

\begin{lemma}
  If an outerplanar graph has at least five vertices and is biconnected, then it does not have a pair of alike vertices.
\end{lemma}

\begin{proof}
  \label{outerplanar}
  Assume, for the sake of contradiction, that there exists a pair of alike vertices. We will do casework on the number of their common neighbors.
  \begin{itemize}
    \item Both vertices cannot have only one common neighbor, as each vertex in the graph has degree at least two, following from the fact the graph is biconnected.
    \item If both vertices have at least three common neighbors, it would follow that $K_{2,3}$ is a subgraph. This contradicts the well-known characterization that no subdivision of $K_{2,3}$ is outerplanar.
    \item It is then the case that both vertices are connected to two common neighbors. As there are at least five vertices, there must exist another vertex $v$, distinct from the alike vertices.

    From biconnectedness, it follows that there exists a path from $v$ to either neighbor vertex that does not pass through the other neighbor vertex. Contracting these paths produces $K_{2,3}$, contradicting once again the well-known characterization that no subdivision of $K_{2,3}$ is outerplanar.\qedhere
  \end{itemize}
\end{proof}

To present the subsequent lemma, we must define the graph of the \emph{bipyramid}. It is the graph with vertex set $\{a_1, a_2\} \cup \{b_1, b_2, \ldots, b_n\}$ and edge set $$\{a_ib_j | i \in [1, 2], j \in [1, n]\} \cup \{b_ib_{i+1} | i \in [1, n], b_{n+1} = b_1\}.$$

\begin{lemma}
  If a maximal planar graph contains a pair of alike vertices, it is a bipyramid.
\end{lemma}

\begin{proof}
  \label{maximalplanar}
  Suppose that there exists a pair of alike vertices $a, b$ in the graph and consider the set $S$ of their common neighbors. Let $s, t \in S$ be two distinct vertices such that the region $asbt$ does not contain any other vertices in $S$ inside. (If there were, we can pick that vertex instead of $s$ and repeat the process.)

  We claim that the region enclosed by $asbt$ does not contain any vertices not from $S$. Indeed, if there were several vertices in this region $v_1, v_2, \ldots, v_i$, then none of these can be connected to $a$ or $b$. Thus $sv_1v_2\cdots v_it$ is a path, for some numbering of the vertices. But then this is a subdivision of the edge $st$, and since neither of the vertices are connected to $a$ or $b$ we cannot produce a face with three sides.

  Thus $asbt$ contains neither vertices from $S$ and not from $S$, so there are no other vertices inside. Since $ab$ cannot be an edge, $st$ must be an edge, as the graph is maximal planar. Repeat this argument for each such region, and we see that all the vertices of $S$ are contained in a single cycle, producing the bipyramid with vertex set $\{a, b\} \cup S$.

  Thus this bipyramid is a subgraph of the maximal planar graph. However, we cannot add any more vertices to any of the regions through a similar argument: if we have vertices $v_1, v_2, \ldots, v_i$ in the region enclosed by $ast$, none of these can be connected to $a$, so $sv_1v_2\cdots t$ is a path for some numbering of the vertices. This is a subdivision of the edge $st$, and since neither of these vertices can be connected to $a$, there must be no other vertices in the region $ast$.
\end{proof}

\section{Bounds through coloring}

The main result of this section is the following lemma, a corollary of Lemma \ref{selection}.

\begin{lemma}
  \label{coloring}
  For a graph $G$ with $n$ vertices, chromatic number $\chi$, metric dimension $\beta$, partitioned by alikeness to $s$ equivalence sets, we have $$\beta \leq \floor{\del{\frac{\chi-1}{\chi}}n} + n - s.$$
\end{lemma}

\begin{proof}
  Color $G$ with $\chi$ colors, and select all colors in the working set except for the most frequent color. This selects at most $\floor{\del{\frac{\chi - 1}{\chi}}n}$ vertices. Of the most common color, which is partitioned by alikeness to several equivalence sets, select all except one. This selects at most $n - s$ vertices. Finally, by Lemma \ref{selection}, we are done.
\end{proof}

With the lemma, we can prove the two remaining main results. 

\begin{proof}[Proof of Theorem \ref{second}]
  An outerplanar graph can be divided into biconnected components, connected by paths. Call a \emph{good} component one that has at least five vertices, and a \emph{bad} one otherwise. Each good component has no alike vertices by Lemma \ref{outerplanar}, and no path has alike vertices. We only need to consider the bad components.

  A bad component with three vertices cannot have two alike vertices as it must be $C_3$. The only case left is when the bad component has four vertices, so it is either $C_4$ or the diamond. In the case of $C_4$ we place a temporary edge connecting the alike vertices, and in the case of the diamond we temporarily remove the diagonal edge and place a temporary edge connecting the alike vertices.

  We then three-color the graph and apply Lemma \ref{coloring}, selecting at most $\floor{\frac{2n}{3}}$ vertices from each good component. We then restore the original graph and show that this selection of vertices is also a resolving set. Indeed, we only need to consider the temporary edges we added and removed, and show that they resolve the bad components:

  In the case of $C_4$, the two alike vertices will have different colors, so at least one of them will be selected. In the case of the diamond, the two alike vertices will have different colors as well, so at least one of them will be selected. It follows that the remaining vertices in the bad component are resolved by Lemma \ref{selection}. At most half of the vertices are selected in the bad component, which is still less than $\floor{\frac{2n}{3}}$.
\end{proof}

\begin{proof}[Proof of Theorem \ref{third}]
  By the four-color theorem, it is possible to four-color the maximal planar graph. If it is not a bipyramid, applying Lemma \ref{coloring} gives the desired bound.

  For the case when the graph is a bipyramid, we show that the metric dimension of a bipyramid is $\beta = \floor{\frac{2n}5} + 1$ for $n \geq 5$, which satisfies the required bound. Let it have the vertex set $\{a_1, a_2\} \cup \{b_1, b_2, \ldots, b_n\}$ and edge set $$\{a_ib_j | i \in [1, 2], j \in [1, n]\} \cup \{b_ib_{i+1} | i \in [1, n], b_{n+1} = b_1\}.$$

  Note that only $a_1$ and $a_2$ can resolve each other by Lemma \ref{selection}, thus we need to select one of them. Without loss of generality, suppose the resolving set is $\{a_1, b_{n_1}, b_{n_2}, \ldots, b_{n_k}\}.$ We make several observations:

  \begin{enumerate}
    \item Any $b_i$ has distance $1$ to its adjacent vertices, and distance $2$ to all other vertices.
    \item The difference $n_{i+1} - n_i$ is at most $4$. If it was greater than $5$, the vertices between $b_{n_i}$ and $b_{n_{i+1}}$ adjacent to neither would have distance $2$ to both, and by the first observation, it would have distance $2$ to all other vertices. Thus it cannot be resolved.
    \item There is at most one $i$ such that $n_{i+1} - n_i$ is equal to $4$. If there were two such $i$, by the first observation the vertices between $b_{n_i}$ and $b_{n_{i+1}}$ would both have distance $2$ to all non-adjacent vertices in the resolving set, and cannot be resolved.
    \item It is impossible for $n_{i+1} - n_i = n_i - n_{i-1} = 3$. Otherwise the vertices adjacent to $b_{n_i}$ would have distance $1$ to $b_{n_i}$ and distance $2$ to all others, and cannot be resolved.
  \end{enumerate}

  We first prove the lower bound. By observation 3 there is at most one $i$ such that $n_{i+1} - n_i$ is equal to $4$. Partition the vertices into groups with at most five vertices from $a_{n_i}$ to $a_i$ and from $a_{n_{i+1}}$ to $a_n$.

  By the pigeonhole principle, if less than $\floor{\frac{2n}{5}} + 1$ vertices are selected, there will be one set with only one vertex, and by observations 2 and 4 this is impossible. This is establishes the lower bound. For the upper bound, select the vertices $b_1, b_{5k}, b_{5k+2}, b_n$ for $k \in \NN$ as the resolving set, which can be easily verified to work.
\end{proof}

\section{Concluding remarks}

Note that for planar graphs in general we have $2 \leq \beta \leq n - 2$ for $n \geq 5$. The lower bound is achieved through $P_n$ and the upper bound through $K_{2, n-2}$. Thus it is impossible to provide bounds of the form $kn$ for planar graphs in general. 

We present the following conjecture for maximal planar graphs, based on small cases and the metric dimension of bipyramids computed in the proof of Theorem \ref{third}.

\begin{conjecture*}
  For maximal planar graphs, $\beta = \floor{\frac{2n}{5}} + O(1)$.
\end{conjecture*}

\bibliographystyle{amsplain}

\end{document}